\documentclass[a4paper,12pt]{article} 
\usepackage[utf8]{inputenc}
\usepackage[T1]{fontenc}
\usepackage{amsmath,amssymb,amsthm}
\usepackage[]{authblk}
\usepackage{bbm}
\usepackage{mathrsfs}
\usepackage[numeric,initials,nobysame,msc-links,abbrev]{amsrefs}
\renewcommand{\eprint}[1]{\href{https://arxiv.org/abs/#1}{arXiv:#1}}
\newcommand{\pageafter}[1]{#1~pp.}
\BibSpec{article}{%
+{} {\PrintAuthors} {author}
+{,} { \textit} {title}
+{.} { } {part}
+{:} { \textit} {subtitle}
+{,} { \PrintContributions} {contribution}
+{.} { \PrintPartials} {partial}
+{,} { } {journal}
+{} { \textbf} {volume}
+{} { \PrintDatePV} {date}
+{,} { \issuetext} {number}
+{,} { \pageafter} {pages}
+{,} { } {status}
+{,} { \PrintDOI} {doi}
+{,} { available at \eprint} {eprint}
+{} { \parenthesize} {language}
+{} { \PrintTranslation} {translation}
+{;} { \PrintReprint} {reprint}
+{.} { } {note}
+{.} {} {transition}
+{} {\SentenceSpace \PrintReviews} {review}
}
\BibSpec{collection.article}{%
+{} {\PrintAuthors} {author}
+{,} { \textit} {title}
+{.} { } {part}
+{:} { \textit} {subtitle}
+{,} { \PrintContributions} {contribution}
+{,} { \PrintConference} {conference}
+{} {\PrintBook} {book}
+{,} { } {booktitle}
+{,} { \PrintDateB} {date}
+{,} { \pageafter} {pages}
+{,} { } {status}
+{,} { \PrintDOI} {doi}
+{,} { available at \eprint} {eprint}
+{} { \parenthesize} {language}
+{} { \PrintTranslation} {translation}
+{;} { \PrintReprint} {reprint}
+{.} { } {note}
+{.} {} {transition}
+{} {\SentenceSpace \PrintReviews} {review}
}
\usepackage{verbatim}
\usepackage{mathtools}
\usepackage{accents}
\usepackage{color}
\usepackage{bbm}
\usepackage{subcaption}
\usepackage{enumitem}
\usepackage{fullpage}

\usepackage{tikz} 

\usepackage[capitalise]{cleveref}

\usetikzlibrary{shapes}
\usetikzlibrary{arrows}
\usetikzlibrary{shapes.misc}
\usetikzlibrary{decorations.pathreplacing}
\usetikzlibrary[patterns]
\usetikzlibrary{hobby}
\usetikzlibrary{arrows.meta}

\tikzset{cross/.style={cross out, draw=black, minimum size=2*(#1-\pgflinewidth), inner sep=0pt, outer sep=0pt},
cross/.default={1pt}}

\setlist[itemize]{leftmargin=*}
\setlist[enumerate]{leftmargin=*,label=(\roman*),ref=(\roman*)}

\newtheorem{thm}{Theorem}
\crefname{thm}{Theorem}{Theorems}
\newtheorem{cor}[thm]{Corollary}
\crefname{cor}{Corollary}{Corollaries}
\newtheorem{lem}[thm]{Lemma}
\crefname{lem}{Lemma}{Lemmas}
\newtheorem{prop}[thm]{Proposition}
\crefname{prop}{Proposition}{Propositions}
\newtheorem{conj}[thm]{Conjecture}
\crefname{conj}{Conjecture}{Conjectures}

\crefname{ques}{Question}{Questions}
\theoremstyle{definition}

\crefname{defn}{Definition}{Definitions}
\newtheorem{defi}[thm]{Definition}
\crefname{defi}{Definition}{Definitions}
\newtheorem{rem}[thm]{Remark}
\crefname{rem}{Remark}{Remarks}

\crefname{ex}{Example}{Examples}

\crefname{obs}{Observation}{Observations}

\crefname{claim}{Claim}{Claims}

\crefname{ass}{Assumption}{Assumptions}

\numberwithin{thm}{section}

\newcommand{\cA}{\ensuremath{\mathcal A}}

\newcommand{\cE}{\ensuremath{\mathcal E}}

\newcommand{\cS}{\ensuremath{\mathcal S}}
\newcommand{\cT}{\ensuremath{\mathcal T}}
\newcommand{\cU}{\ensuremath{\mathcal U}}

\newcommand{\bbH}{{\ensuremath{\mathbb H}} }

\newcommand{\bbP}{{\ensuremath{\mathbb P}} }

\newcommand{\bbR}{{\ensuremath{\mathbb R}} }

\newcommand{\bbZ}{{\ensuremath{\mathbb Z}} }

\let\oldd\d
\renewcommand{\d}{{\ensuremath{\delta}}}

\newcommand{\e}{{\ensuremath{\varepsilon}}}

\let\oldk\k
\renewcommand{\k}{{\ensuremath{\kappa}}}

\let\oldl\l
\renewcommand{\l}{{\ensuremath{\lambda}}}
\let\oldL\L
\renewcommand{\L}{{\ensuremath{\Lambda}}}

\let\oldo\o
\renewcommand{\o}{{\ensuremath{\omega}}}
\let\oldO\O
\renewcommand{\O}{{\ensuremath{\Omega}}}
\newcommand{\p}{{\ensuremath{\pi}}}
\let\oldr\r
\renewcommand{\r}{{\ensuremath{\rho}}}
\newcommand{\s}{{\ensuremath{\sigma}}}
\let\oldS\S
\renewcommand{\S}{{\ensuremath{\Sigma}}}
\let\oldt\t
\renewcommand{\t}{{\ensuremath{\tau}}}
\let\oldu\u
\renewcommand{\u}{{\ensuremath{\upsilon}}}

\renewcommand{\>}{\rangle}

\newcommand{\1}{{\ensuremath{\mathbbm{1}}} }

\newcommand{\pc}{\ensuremath{p_{\mathrm{c}}} }

\renewcommand{\leq}{\leqslant}
\renewcommand{\geq}{\geqslant}
\renewcommand{\le}{\leqslant}
\renewcommand{\ge}{\geqslant}
\renewcommand{\to}{\rightarrow}

\newcommand{\bt}{\begin{thm}}
\newcommand{\et}{\end{thm}}
\newcommand{\bl}{\begin{lem}}
\newcommand{\el}{\end{lem}}
\newcommand{\bp}{\begin{prop}}
\newcommand{\ep}{\end{prop}}
\newcommand{\bcor}{\begin{cor}}
\newcommand{\ecor}{\end{cor}}
\newcommand{\br}{\begin{rem}}
\newcommand{\er}{\end{rem}}
\newcommand{\bcon}{\begin{conj}}
\newcommand{\econ}{\end{conj}}

\newcommand{\bpro}{\begin{proof}}
\newcommand{\epro}{\end{proof}}

\newcommand{\be}{\begin{equation}}
\newcommand{\ee}{\end{equation}}

\expandafter\mathchardef\expandafter\varphi\number\expandafter\phi\expandafter\relax
\expandafter\mathchardef\expandafter\phi\number\varphi

\newcommand{\eps}{\varepsilon}

\newcommand{\sig}{\sigma}
\newcommand{\Sig}{\Sigma}

\newcommand{\Ei}{{\cal E}}

\renewcommand{\P}{{\mathbb P}}
\newcommand{\R}{{\mathbb R}}
\newcommand{\Z}{{\mathbb Z}}

\begin{document}

\title{Subcritical bootstrap percolation via Toom contours\footnote{This work was supported by ERC Starting Grant 680275 ``MALIG.''}}
\renewcommand\Affilfont{\small}
\author{Ivailo Hartarsky\thanks{\textsf{hartarsky@ceremade.dauphine.fr}} } 
\author{R\'eka Szab\'o\thanks{\textsf{szabo@ceremade.dauphine.fr}}}
\affil{CEREMADE, CNRS, Universit\'e Paris-Dauphine, PSL University\protect\\Place du Mar\'echal de Lattre de Tassigny, 75016 Paris, France}
\date{\vspace{-0.25cm}\today}

\maketitle

\begin{abstract}
In this note we provide an alternative proof of the fact that subcritical bootstrap percolation models have a positive critical probability in any dimension. The proof relies on a recent extension of the classical framework of Toom. This approach is not only simpler than the original multi-scale renormalisation proof of the result in two and more dimensions, but also gives significantly better bounds. As a byproduct, we improve the best known bounds for the stability threshold of Toom's North-East-Center majority rule cellular automaton.
\end{abstract}
\noindent\textbf{MSC2020:} Primary 60K35; Secondary 60C05, 82C20
\\
\textbf{Keywords:}  bootstrap percolation, Toom rule, North-East-Center majority, critical probability, stability threshold, Toom contour

\section{Introduction}
Bootstrap percolation is a statistical mechanics model initially introduced to model magnetic materials at low temperature \cite{Chalupa79}. It has proved useful for studying the dynamics of the Ising and kinetically constrained models (see \cite{Morris17} for a review), but it has also been used more directly to model e.g.\ the spread of news on a social network (see \cite{Banerjee20} for a survey). Since its introduction many variants of this cellular automaton have been investigated. This led to the study of the universality classes of bootstrap percolation models defined by the various possible behaviours with a random i.i.d.\ initial condition. This classification was completed in two dimensions \cite{Balister16,Bollobas15,Bollobas14} and was recently extended to higher dimensions by Balister, Bollob\'as, Morris and Smith \cite{Balister22,Balister22a,Balister22b}.

An apparently unrelated area of study is the one of random perturbations of monotone cellular automata. Perhaps the most central result in this domain is due to Toom \cite{Toom80}. Motivated by reliable computing via celullar automata, he proved an efficient necessary and sufficient condition for a monotone cellular automaton to be resistant to noise.

Recently, the first author \cite{Hartarsky22sharpness} noted that these two domains are closely related. In particular, he showed the classical \cite{Toom80} to be equivalent to part of the main result of \cite{Balister22}, still in preparation at the time. Around the same time Swart, Toninelli and the second author \cite{Swart22} extended Toom's framework further to encompass attractive probabilistic cellular automata instead of deterministic ones. The goal of the present work is to adapt their tools to provide an alternative proof of the full universaility result of \cite{Balister22}, which is also quantitatively more efficient and establishes yet a new connection between the two settings.
\subsection{Bootstrap percolation}
\label{S:bootstrap}
Fix a dimension $d\ge 2$. A \emph{bootstrap percolation model} on $\mathbb Z^d$ is a monotone cellular automaton specified by an \emph{update family} $\mathcal U$, that is, a finite family of finite subsets of $\mathbb Z^d\setminus\{o\}$ ($o$ denotes the origin of $\bbZ^d$). We start from an initial configuration $x\in\O:=\{0,1\}^{\mathbb Z^d}$. At each time step the process evolves according to a local rule. Denoting by $X_t$ the set of vertices in state 0 at time $t\geq 0$, the set $X_{t+1}$ is defined by
\be\label{eq:BPevolution}
X_{t+1} := X_{t} \cup\left\{ i\in\mathbb Z^d\colon \exists U \in\mathcal U \mbox{ such that } i+U \subset X_t\right\}.
\ee
That is, a site $i$ becomes 0 if and only if it was already in state 0 or there exists a finite $U\subset\bbZ^d\setminus\{o\}$ in the update family $\cU$ such that all elements of $i+U$ are in state 0. Note that randomness is only involved in the state of the configuration at time 0, after that the evolution of the process is deterministic. For an initial configuration $X_0=X$ we denote by $[X] = \bigcup_{t\geq 0}X_t$ the \emph{closure} of $X$ and say that the process \emph{percolates} if $[X]=\mathbb Z^d$. Let $\mathbb P_p$ denote the law of the process starting from an initial configuration where each site is in state 0 with probability $p$ and in state 1 with probability $1-p$ independently from each other. We define the \emph{critical parameter}
\[\pc(\cU):=\inf \left\{p\in[0,1]\colon  \mathbb P_p\left([X]=\mathbb Z^d\right)=1\right\}.\]
Note that by ergodicity $\mathbb P_p([X]=\mathbb Z^d)\in\{0,1\}$ for all $p\in[0,1]$. The main question we would like to address is for which $\mathcal U$ the above phase transition is non-trivial in the sense $\pc(\cU)>0$. The answer was suggested by Balister, Bollob\'as, Przykucki and Smith \cite{Balister16} and requires a few more notions to state.

We denote by $S^{d-1}$ the unit sphere and by $\langle\cdot,\cdot\rangle$ the scalar product in $\R^d$. For each unit vector $u\in S^{d-1}$ we let $\mathbb H_u:=\{v\in \R^d\colon \langle v,u\rangle < 0\}$ denote the open half-space whose boundary is perpendicular to $u$. We say that a direction $u\in S^{d-1}$ is \emph{stable}, if $[\mathbb H_u\cap \Z^d] =\mathbb H_u\cap \mathbb Z^d$, and denote by $\mathcal S\subset S^{d-1}$ the set of all stable directions. We say that a direction $u$ is \emph{strongly stable}, if it is in the interior of $\mathcal S$. We say that an update family is \emph{subcritical}, if every hemisphere of $S^{d-1}$ contains a strongly stable direction.

Our goal is to provide a simple proof of the following result recently established in \cite{Balister22} (see \cite{Balister22b} for the converse).

\begin{thm}
	\label{T:bootstrap}
	If $\cU$ is subcritical, then $\pc(\cU)>0$.
\end{thm}
\begin{rem}
	Following \cite[Remark 1.5]{Hartarsky20II} (see also \cite[Section 2.6]{Swart22}), let us note that Theorem~\ref{T:bootstrap} applies equally well to a space-time inhomogeneous version of bootstrap percolation. Namely, we apply a random update family at each space-time point chosen independently among finitely many families $\cU_1,\dots,\cU_n$. The model is called \emph{subcritical} if the single update family $\cU=\bigcup_{j=1}^{n}\cU_j$ is subcritical. In this case for some fixed $p>0$ and initial state with law $\bbP_p$, the process does not percolate a.s.\ w.r.t.\ the inhomogeneity.
\end{rem}

The first instances of Theorem~\ref{T:bootstrap} were established already by Schonmann \cite{Schonmann90,Schonmann92} in the 1990s. Theorem~\ref{T:bootstrap} restricted to $d=2$, was proved by Balister, Bollob\'as Przykucki and Smith \cite{Balister16}, using a rather involved multi-scale renormalisation. They conjectured Theorem~\ref{T:bootstrap} \cite[Conjecture~16]{Balister16} and suggested that modulo further technical difficulties they expect their approach to work in higher dimensions. This conjecture was reiterated in \cite[Conjecture 1.6]{Morris17} and recently verified by Balister, Bollob\'as, Morris and Smith \cite{Balister22} by the same technique. Meanwhile, using Toom's result \cite{Toom80} (see Section \ref{subsec:Toom}), the first author \cite{Hartarsky22sharpness} proved Theorem \ref{T:bootstrap} for $\mathcal U$ contained in some half-space $\mathbb H_u$, that is, for every  $U\in\mathcal U$ it holds that $U\subset\mathbb H_u$.

As already noticed by Schonmann \cite{Schonmann90}, oriented site percolation can be viewed as a subcritical bootstrap percolation model (see \cite{Hartarsky22sharpness} for a generalisation of this fact). Quantitative rigorous bounds on $\pc$ for this model had been obtained much earlier (see \cite{Durrett84} for an overview).

For general models, particularly non-oriented ones, bounds are much more difficult to obtain. For this reason \cite{Balister16} introduced a benchmark two-dimensional subcritical model called \emph{directed triangular bootstrap percolation} (DTBP) given by
\be
\label{eq:DTBP}
\cU^{\mathrm{DTBP}}:=\left\{\{(1,0),(0,1)\}, \{(-1,-1),(0,1)\}, \{(-1,-1),(1,0)\}\right\}
\ee
(see Fig.~\ref{fig:DTBP}). The proof of Theorem~\ref{T:bootstrap} in \cite{Balister16} gave the lower bound in
\begin{equation}
\label{eq:BBPS:DTBP}
10^{-101}<\pc\left(\cU^{\mathrm{DTBP}}\right)<0.2452,
\end{equation}
while the upper bound was proved by the first author \cite{Hartarsky21} also by a general approach. Naturally, the lower bound in Eq.~\eqref{eq:BBPS:DTBP} is more disappointing and \cite[Question 17]{Balister16} asks for improving that. Our proof of Theorem \ref{T:bootstrap} provides the following improvement, still far from the nonrigorous numerical estimate $\pc\left(\cU^{\mathrm{DTBP}}\right)\approx 0.118$ put forward in \cite{Balister16}.

\begin{thm}
	\label{th:DTBP}
	For the DTBP update family given in Eq.~\eqref{eq:DTBP}, $\pc\left(\mathcal U^{\mathrm{DTBP}}\right)> 2.8\cdot 10^{-6}$.
\end{thm}

\begin{figure}
	\centering
	\begin{subfigure}{8cm}
		\centering
		\includegraphics[width=7cm]{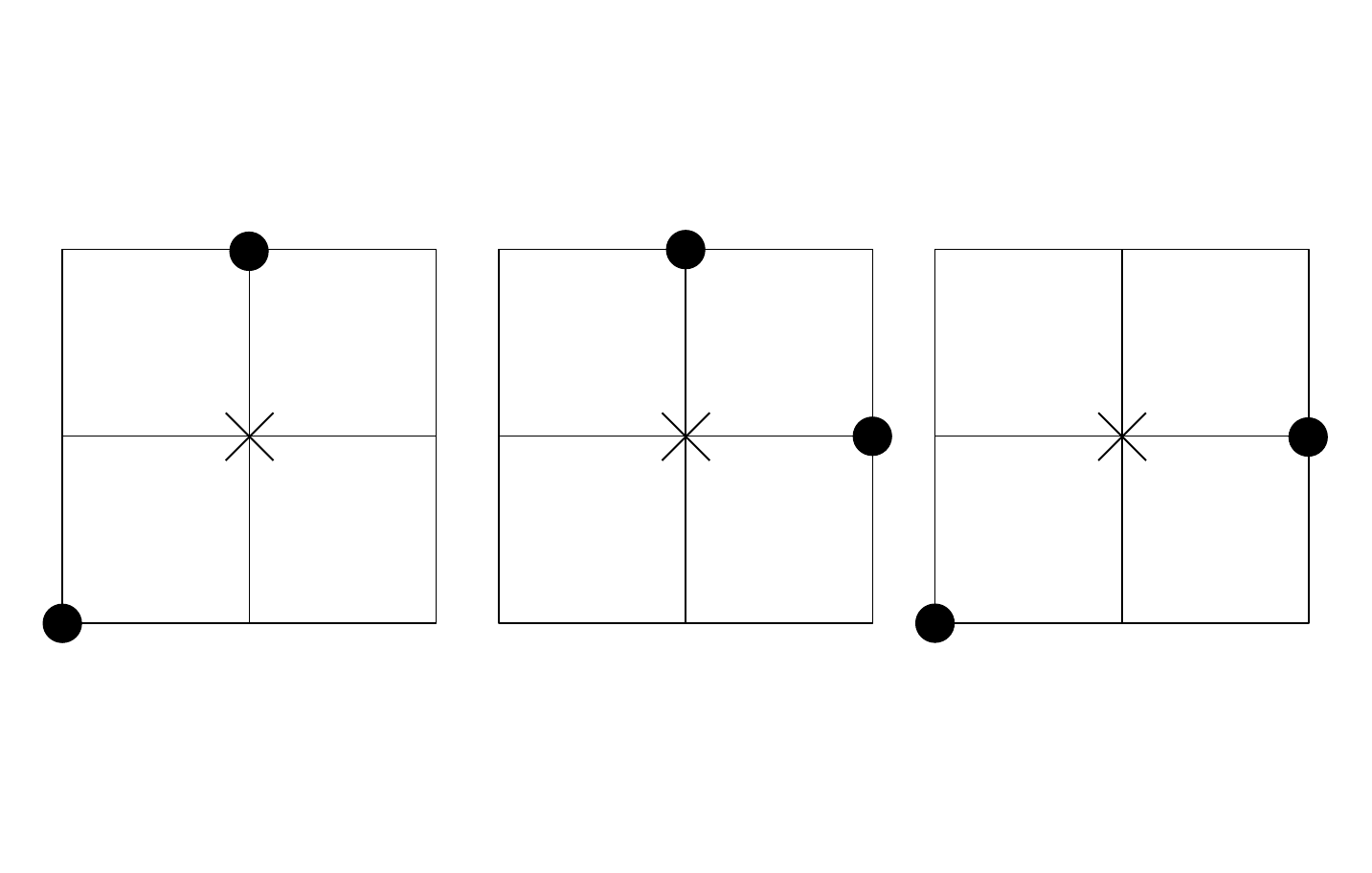}
		\caption{\raggedright Family $\cU$ coinciding with $\{A_1,A_2,A_3\}$. The cross marks $o$, the solid dots mark elements of $A_s$ for $s\in\S$.}
	\end{subfigure}
	\quad\begin{subfigure}{4.55cm}
		\centering
		\includegraphics[width=4.5cm]{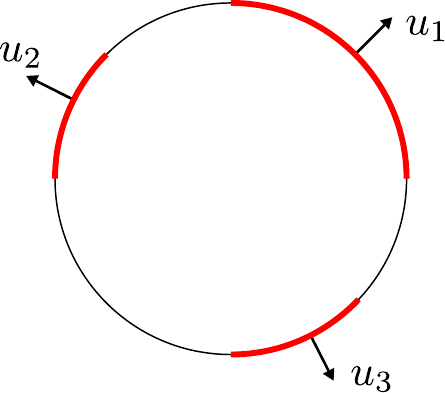}
		\caption{Stable directions.}
	\end{subfigure}
	\caption{The DTBP example with parameters as in~Eq.~\eqref{eq:DTBPsets}.}
	\label{fig:DTBP}
\end{figure}

\subsection{Perturbed cellular automata}
\label{subsec:Toom}
Toom \cite{Toom80} studied random perturbations of monotone cellular automata. More precisely, we are given some map $\phi:\O\to\{0,1\}$ depending on finitely many coordinates of the input and such that $\phi(x)\le\phi(y)$ whenever $x\le y$ for the coordinate-wise order. We start from the configuration $x_0$ equal to $1$ everywhere and let $x_{t+1}(i)=\phi(x_t(\cdot+i))$ with probability $1-p$ and $x_{t+1}(i)=0$ with probability $p$.\footnote{In fact, Toom studied a more general type of noise, but the monotonicity property allows one to reduce their study to the simple one we focus on.} We then set
\[\pc(\phi):=\sup\left\{p\in[0,1],\liminf_{t>0}\bbP_p\left(x_t(o)=1\right)>0\right\}.\]
One is interested for which $\phi$ we have $\pc(\phi)>0$, which suggests that the corresponding cellular automaton does not completely lose memory of the initial state when subjected to noise. In order to answer this, we need one more notion. 

We say that $\phi$ is an \emph{eroder} if any initial condition with only finitely many sites in state $0$ eventually becomes all $1$ in the absence of noise ($p=0$). Toom \cite{Toom80} famously proved that $\phi$ is an eroder if and only if $\pc(\phi)>0$. The hard implication ($\pc(\phi)>0$ if $\phi$ is an eroder) was shown in \cite{Hartarsky22sharpness} to follow from Theorem~\ref{T:bootstrap}. Conversely, our proof of Theorem~\ref{T:bootstrap} relies on an improvement of the method of \cite{Toom80} recently revisited and generalised by Swart, Toninelli and the second author \cite{Swart22} (see Section~\ref{sec:Toom}). Restricted or full versions of Toom's result have been proved alternatively in \cite{Berman88,Bramson91,Gacs95,Gacs21,Gacs88,Swart22} (see \cite[Section 1.4]{Swart22} for more detailed background). Most of them rely on a Peierls argument.

An important two-dimensional example is the Toom North-East-Center majority rule: \begin{equation}
\label{eq:def:Toom}\phi^{\mathrm{NEC}}(x):=\1_{x(o)+x((1,0))+x((0,1))\ge 2}.\end{equation}
It was introduced in \cite{Vasilev69} and $\pc\left(\phi^{\mathrm{NEC}}\right)>0$ can be recovered from \cite{Toom74}. The best explicit bound $\pc\left(\phi^{\mathrm{NEC}}\right)\ge3^{-21}\approx 9.6\cdot 10^{-11}$ was obtained recently \cite{Swart22}. It turns out that the proof of Theorem \ref{th:DTBP} leads to the following improvement to be compared with the nonrigorous numerical estimate $\pc\left(\phi^{\mathrm{NEC}}\right)\approx 0.053$ \cite{Swart22}.
\begin{thm}
	\label{th:Toom}
	For the Toom rule of \eqref{eq:def:Toom} we have $2.8\cdot 10^{-6}<\pc\left(\phi^{\mathrm{NEC}}\right)$.
\end{thm}
\begin{rem}
	\label{rem:Toom}
	We are unaware of upper bounds on $\pc\left(\phi^{\mathrm{NEC}}\right)$, but the bound $0.3118$ follows from comparison with oriented site percolation and \cite{Gray80}. Applying the first author's work \cite{Hartarsky21,Hartarsky22sharpness}, one can prove $\pc\left(\phi^{\mathrm{NEC}}\right)<0.2452$. A sketch of the argument is provided in the appendix.
\end{rem}

\section{Preliminaries}
\label{sec:preliminaries}
For the rest of the paper we fix a subcritical update family $\cU$. While $\cU$ specifies sets of 0-s sufficient for the origin to become 0, it will be more convenient for us to work with an alternative representation of the model. Namely, in terms of sets of 1-s preventing $o$ to become 0 on the next time step. More precisely, let
\[\mathcal A:=\left\{\{i_1,\dots, i_n\}\colon  \forall j\in\{1,\dots,n\},i_j\in U_j\right\},\]
where $\cU=\{U_1,\dots,U_n\}$. We next show that for subcritical models one can find a suitable set of directions $u_i$ and sets $A_i\in\mathcal A$ so that $u_j$ `points towards' $A_j$.

\begin{lem}
	\label{lem:directions}
	For any subcritical $\cU$ there exists an integer $\s\in\{2,\dots,d+1\}$, strongly stable directions $u_1,\dots,u_{\s}\in S^{d-1}$, real coefficients $\l_1,\dots,\l_\s\in(0,1)$ and sets $A_1,\dots,A_\s\in\cA$ such that
	\begin{equation}
	\label{eq:polar}\sum_{j=1}^\s\l_ju_j=0
	\end{equation}
	and $A_j\subset \bbH_{-u_j}$ for all $j\in\{1,\dots,\s\}$.
\end{lem}
\begin{proof}
	Let $\ring\cS$ be the set of strongly stable directions. Let $\hat\cS$ be the set of $u\in\ring\cS$ such that for all $U\in\cU$ and $i\in U$ we have $\<i,u\>\neq0$. Assume that $o$ is not in the interior of the convex envelope of $\hat\cS$. Then by the finite dimensional Hahn--Banach separation theorem there exists an open hemisphere $H$ disjoint from $\hat\cS$. Yet, $\cU$ is subcritical, so $\ring S\cap H\neq\varnothing$. But this is a contradiction, since $\ring\cS\setminus\hat\cS$ has empty interior in $S^{d-1}$. 
	
	Thus, there exist directions in $\hat\cS$ whose convex combination is $o$. Moreover, by Carath\'eodory's theorem, we may select at most $d+1$ of these directions, so that the same holds, yielding~Eq.~\eqref{eq:polar}.
	
	Observe that a direction $u\in S^{d-1}$ is stable if and only if $A\cap\bbH_u=\varnothing$ for some $A\in\cA$. But if $u\in\hat\cS$ (and not just $u\in\ring\cS$) this is equivalent to the existence of $A\subset\bbH_{-u}$.
\end{proof}
For the rest of the paper we fix $\s$, $u_s$, $\l_s$ and $A_s$ for $s\in\S:=\{1,\dots,\s\}$ as in Lemma~\ref{lem:directions}. Their role is that if we perform a walk with steps in $A_s$, we end up drifting away in direction $u_s$. Indeed, such long walks will play an important part in our proof. Eq.~\eqref{eq:polar} guarantees that if we consider these walks for each $s\in\Sigma$, their endpoints globally spread out in space linearly, even though some of them may remain together. In order to formalise and quantify this effect, we need some more notation. 

Consider the linear forms $L_s:\bbR^d\to\bbR$ \begin{equation} L_s(i):=\l_s\<i,u_s\> \quad (s\in\S). \label{eq:def:L}
\end{equation}
Further let 
\begin{align}
\e&{}:=\min_{s\in\S}\min_{i\in A_s}L_s(i)>0,& 
R&{}:=-\sum_{s\in\S}\min_{i\in A}L_s(i),
\label{eq:def:R:eps}
\end{align}
where $A=\bigcup_{s\in\S} A_s\subset \bigcup_{U\in\cU}U$. Thus, $\e$ is the `minimum drift' and $R$ is a `total negative drift, if one makes a step in the wrong direction'. Note that $\e>0$, as $A_s\in \bbH_{-u_s}$ for all $s\in\Sig$. For our DTBP example (see Fig.~\ref{fig:DTBP}) we simply set $\s=3$ and
\begin{equation}\label{eq:DTBPsets}
\begin{aligned}
A_1&{}:=\{(1,0),(0,1)\}&u_1&{}:=\frac{1}{\sqrt 2}(1,1)&\l_1&{}:=\sqrt 2,\\
A_2&{}:=\{(-1,-1),(0,1)\}\quad&u_2&{}:=\frac{1}{\sqrt 5}(-2,1)\quad&\l_2&{}:=\sqrt 5,\\
A_3&{}:=\{(-1,-1),(1,0)\}&u_3&{}:=\frac{1}{\sqrt 5}(1,-2)&\l_3&{}:=\sqrt 5.
\end{aligned}
\end{equation}
These do verify Lemma~\ref{lem:directions} and the constants of Eq.~\eqref{eq:def:R:eps} are
\begin{align}\label{eq:DTBPconstants}
\e&{}=1,&R&{}=6. 
\end{align}

\section{Toom contours}
\label{sec:Toom}
In the present section we closely follow~\cite{Swart22} adapted to our bootstrap percolation setting. We refer to that work for more details, but the main idea is to construct a graph which explains how 0-s propagate to reach a given space-time point. Roughly speaking, if a vertex $i$ is in state 0 at time $t$, then either $i\in X_0$, or there is a vertex $j_s\in i+A_s$ in state 0 at time $t-1$ for each $s\in\Sigma$. We can then choose similar vertices for each $j_s$ that is not in $X_0$, and so on. This way, starting from $i$, for each $s\in\Sigma$ we define a walk with steps in $A_s$ ending at a vertex in $X_0$. We use these walks to construct Toom contours.

We define a \emph{directed graph} as a couple $(V,\vec E)$ where $V$ is a set of vertices and $\vec E$ is a set of directed edges that is a subset of $V\times V$. Let
\begin{align*}
\vec E_{\rm in}(v):={}&\left\{(u,v)\in\vec E\right\},&
\vec E_{\rm out}(v):={}&\left\{(v,w)\in\vec E\right\}
\end{align*}
denote the sets of directed edges entering and leaving a given vertex $v\in V$, respectively. We further define an \emph{directed graph with $\sig$ types of edges} to be a couple $(V,\Ei)$, where $\Ei=(\vec E_1,\ldots,\vec E_\sig)$ is a sequence of subsets of $V\times V$. We interpret $\vec E_s$ as the set of directed edges of type $s$.

\newcounter{conditions}
\begin{defi}[Toom graph]\label{def:toomgraph}
	A \emph{Toom graph} with $\sig\geq 2$ \emph{charges} is a directed graph with $\sig$ types of edges $(V,\Ei)=(V,(\vec E_1,\ldots,\vec E_\sig))$ such that each vertex $v\in V$ satisfies one of the following four conditions (see the left of Fig.~\ref{fig:toomcontour}):
	\begin{enumerate}
		\item\label{i} $|\vec E_{s,{\rm in}}(v)|=0=|\vec E_{s,{\rm out}}(v)|$ for all $s\in\S$,
		\item\label{ii} $|\vec E_{s,{\rm in}}(v)|=0$ and $|\vec E_{s,{\rm out}}(v)|=1$ for all $s\in\S$, 
		\item\label{iii} $|\vec E_{s,{\rm in}}(v)|=1$ and $|\vec E_{s,{\rm out}}(v)|=0$ for all $s\in\S$, 
		\item\label{iv} there exists $s\in\S$ such that $|\vec E_{s,{\rm in}}(v)|=1=|\vec E_{s,{\rm out}}(v)|$ and $|\vec E_{l,{\rm in}}(v)|=|\vec E_{l,{\rm out}}(v)|=0$ for each $l\in\S\setminus\{s\}$.
		\setcounter{conditions}{\value{enumi}}
	\end{enumerate}
\end{defi}

We set
\begin{align*}
V_\circ:={}&\left\{v\in V\colon \forall s\in\S,
|\vec E_{s,{\rm in}}(v)|=0\right\},\\
V_\star:={}&\left\{v\in V\colon \forall s\in\S,
|\vec E_{s,{\rm out}}(v)|=0\right\},\\
\forall s\in\S\quad V_s:={}&\left\{v\in V\colon 
|\vec E_{s,{\rm in}}(v)|=1=|\vec E_{s,{\rm out}}(v)|\right\}.
\end{align*}
Vertices in $V_\circ,V_\star$, and $V_s$ are called \emph{sources}, \emph{sinks}, and \emph{internal vertices} with \emph{charge} $s$, respectively. Vertices in $V_\circ\cap V_\star$ are called \emph{isolated vertices}. As we can see on the left of Fig.~\ref{fig:toomcontour}, we can imagine that at each source $\sig$ charges emerge, one of each type. Charges then travel via internal vertices of the corresponding charge through the graph until they arrive at a sink, in such a way that at each sink precisely $\sig$ charges arrive, one of each type. It is clear from this description that $|V_\circ|=|V_\star|$, i.e., the number of sources equals the number of sinks.

Let $\vec E:=\bigcup_{s=1}^\sig\vec E_s$ denote the directed edges of all types and $E:=\{\{v,w\}\colon (v,w)\in\vec E\}$ denote the corresponding set of undirected edges. We say that a Toom graph $(V,\Ei)$ is \emph{connected} if the associated undirected graph $(V,E)$ is connected. 

We call a Toom graph with a distinguished source $v_\circ\in V_\circ$ a \emph{rooted Toom graph}. For a rooted Toom graph $(V,\Ei,v_\circ)$ and $s\in\S$, we write
\begin{align*}
\vec E^\star_s:={}&\left\{(v,w)\in\vec E_s:v\in V_s\cup\{v_\circ\}\right\}&\vec E^\star:={}&\bigcup_{s\in\S}\vec E^\star_s,\\
\vec E^\circ_s:={}&\left\{(v,w)\in\vec E_s\colon v\in V_\circ\setminus\{v_\circ\}\right\}&\vec E^\circ:={}&\bigcup_{s\in\S}\vec E^\circ_s.\end{align*}
I.e.\ $\vec E^\star$ is the set of directed edges that have an internal vertex or the root as their starting vertex and $\vec E^\circ$ are all the other directed edges, starting at a source that is not the root.

Our next aim is to define Toom contours, which are connected Toom graphs that are embedded in space-time $\Z^{d+1}$ in a special way.
\begin{defi}[Embedding]\label{def:embedding}
	An \emph{embedding} of a Toom graph $(V,\Ei)$ is a map
	\[\psi:V\to\Z^d\times\Z:v\mapsto\left(\vec\psi(v),\psi_{d+1}(v)\right)\]
	that has the following properties (see Fig.~\ref{fig:toomcontour}):
	\begin{enumerate}
		\setcounter{enumi}{\value{conditions}}
		\item\label{v} $\psi_{d+1}(w)=\psi_{d+1}(v)-1$ for all $(v,w)\in\vec E$,
		\item\label{vi} $\psi(v_1)\neq\psi(v_2)$ for each $v_1\in V_\star$ and $v_2\in V$ with $v_1\neq v_2$,
		\item\label{vii} $\psi(v_1)\neq\psi(v_2)$ for each $s\in\S$ and $v_1,v_2\in V_s$ with $v_1\neq v_2$.
		\setcounter{conditions}{\value{enumi}}
	\end{enumerate}
\end{defi}

We interpret~$\vec\psi(v)$ and~$\psi_{d+1}(v)$ as the space and time coordinates of~$\psi(v)$ respectively. Condition \ref{v} says that directed edges $(v,w)$ of the Toom graph $(V,\cE)$ point in the direction of decreasing time. Condition~\ref{vi} says that sinks do not overlap with other vertices and condition~\ref{vii} says that internal vertices do not overlap with other internal vertices of the same charge. 

Recall the $\bbP_p$-random set $X_0$ and the sets $A_s$ for $s\in\Sig$ given by Lemma~\ref{lem:directions}.

\begin{defi}[Contour]\label{def:toomcontour}
	A \emph{Toom contour} is a quadruple $(V,\Ei,v_\circ,\psi)$ with $(V,\Ei,v_\circ)$ a connected rooted Toom graph and $\psi$ an embedding of it satisfying the following properties (see Fig.~\ref{fig:toomcontour})
	\begin{enumerate}
		\setcounter{enumi}{\value{conditions}}
		\item\label{viii} $\vec\psi(w)=\vec\psi(v)$ for all $(v,w)\in \vec E^\star$ such that $\vec\psi(v)\in \vec\psi(V_\star)$,
		\item\label{x} $\vec\psi(w)-\vec\psi(v)\in A_s$ for all $s\in\S$ and $(v,w)\in\vec E^\star_s$ such that $\vec\psi(v)\not\in \vec\psi(V_\star)$,
		\item\label{xi} $\vec\psi(w)-\vec\psi(v)\in A=\bigcup_{s\in\S} A_s$ for all $(v,w)\in\vec E^\circ$,
		\item\label{ix} $|\{\psi(w)\colon  (v,w)\in\vec E\}|=2$ for all $v\in V_\circ\setminus \{v_\circ\}$,
		\item\label{xii} $\psi_{d+1}(V_\star)=\{0\}$.
	\end{enumerate}
	The Toom contour is \emph{present} in $X_0$ if $\vec\psi(V_\star)\subset X_0$.
\end{defi}
For $s\in\S$, let us call pairs of space-time points of the form $\big((i,t),(i+j,t-1)\big)$ with $j\in A_s$ \emph{type $s$ diagonal segments} and pairs of space-time points of the form $\big((i,t),(i,t-1)\big)$ \emph{vertical segments}. Condition \ref{viii} says that the segments starting from a vertex with the same space coordinate as a sink are vertical. Condition \ref{x} says that edges of charge $s$ starting at internal vertices or the root map to diagonal segments of type $s$ if their starting point does not have the same space coordinate as any sink. Condition \ref{xi} gives that edges from sources other than the root map to diagonal segments of arbitrary type. Condition \ref{ix}, which is only needed to improve our quantitative bounds, ensures that all sources are \emph{forks}: the embeddings of all their $\s$ edges point to exactly two sites (see \cite[Theorem 32]{Swart22}). Together with condition~\ref{viii} it ensures that each source other than the root has a different space coordinate from any sink. Condition \ref{xii} ensures that all sinks are embedded with time coordinate $0$. Finally, the contour is present if sinks are mapped to vertices initially in state 0.

As mentioned before, Toom contours explain how 0-s propagate to reach site $i$ at time $t$. Informally we can imagine it as follows. The embedding of the root is $(i,t)$. The $\sigma$ charges emerging from it correspond to the walks discussed in Section~\ref{sec:preliminaries}, starting at $i$ with steps in $A_s$ and going backwards in time. Their embedding consist of diagonal segments until the walk reaches its endpoint in $X_0$, then the remaining segments up until time 0 are vertical. This ensures that indeed each sink is mapped to a point of $X_0$ at time 0. As we follow these charges, they spread out in space. Whenever they drift `too far' from each other, we add new sources with their corresponding charges to `bridge the gap'. To have some control over the direction of these new charges we are allowed to chose the first step of their embedding in any directions in $A$, while the rest maps to a similar walk in $A_s$. That is why the other sources behave slightly differently from the root, as seen in Conditions~\ref{x} and~\ref{xi}.

The following is~\cite[Theorem 7]{Swart22} in our setting.
\bt[Presence of a Toom contour]\label{T:contour} For any $t\ge 0$ such that $o\in X_t$ we have that a Toom contour rooted at $(o,t)$ is present in $X_0$.
\et

Since the reader may have difficulty reading Theorem~\ref{T:contour} out of \cite{Swart22}, let us explain how to fit our setting into theirs. We define the map $\phi:\O\to\{0,1\}$ as
\[\phi(x):=\begin{cases}0&\exists U\in\cU \text{ such that } x(i)=0\text{ for all }i\in U,\\
1&\text{otherwise}.
\end{cases}\] 
It is not hard to check that $\phi(x)=1$ if and only if there exists $A\in\cA$ such that $x(i)=1$ for all $i\in A$. For every space-time point $(i, t)\in \Z^{d+1}$, we define $\varphi_{i,t}:\O\to\{0,1\}$ by
\begin{equation}\label{eq:varphi}\varphi_{i,t}(x):=\begin{cases}\phi(x) &\text{if } i\in\bbZ^d\setminus X_0, \;t\in\bbZ,\\
0&\text{if }i\in X_0 ,\; t=0,\\
x(o)&\text{if }i\in X_0,\; t\neq 0.\end{cases}\end{equation}
For $X\subset \bbZ^d$ we define $x(X):=\1_{\bbZ^d\setminus X}\in\O$. We then verify from~Eq.~\eqref{eq:BPevolution} that for all $t>0$ and $i\in\bbZ^d$ we have $i\in X_t$ if and only if $\varphi_{i,t}(x(X_{t-1}-i))=0$. Further setting $X_t=\varnothing$ for $t<0$, \cite[Theorem 7]{Swart22} indeed becomes Theorem~\ref{T:contour}. Let us reassure the reader that this notation will not be used further.

\begin{figure}
	\begin{center}
		\includegraphics[width=14cm]{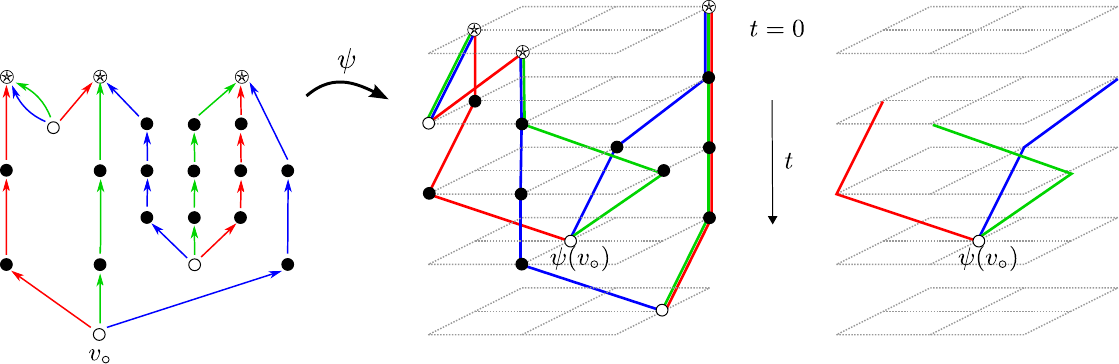}
		\caption{Toom contour for DTBP with $A_1, A_2, A_3$ as in Eq.~\eqref{eq:DTBPsets} and Fig.~\ref{fig:DTBP}. On the left is a Toom graph with $\sig=3$ charges rooted at $v_\circ$, each color representing a different type of edges, in the middle is its embedding in space-time, and on the right is its embedded root shard. Empty dots correspond to sources, while stars denote sinks. The contour is present if sinks belong to $X_0$.}
		\label{fig:toomcontour}
	\end{center}
\end{figure}

\section{Shattering contours}
\label{subsec:shattered}
The core of Toom's Peierls argument is to define contours that we can count efficiently. This is done in \cite{Toom80} for perturbed cellular automata where each space-time point is either 0 or applies the map $\phi$. Even though in \cite{Swart22} the definition of Toom contours is extended for more general models, including bootstrap percolation, in our setting their number explodes. Therefore, we need a more precise notion of a contour,  which is our main novelty together with bounding their number (see Lemma~\ref{lem:counting} below)
. It reflects the fact that the maps $\varphi_{i, t}$ from Eq.~\eqref{eq:varphi} do not depend on $t\in(0,\infty)$, allowing us to shift contours in time. Informally, this new contour is defined only by the space coordinates of the sources and the $\sigma$ walks emerging from them with steps in $A_s$. 

Let $(V,\cE, v_\circ, \psi)$ be a Toom contour. For any $v\in V$  denote by $V_v\subset V$ the set of vertices that can be reached from $v$ in the directed graph $(V, \cE)$ by edges whose embedding is a diagonal segment. 

\begin{defi}[Shard] Given a Toom contour $(V, \cE, v_\circ, \psi)$ and a source  $v\in V_\circ$ we say that $(V_v, \cE_v)$ is a \textit{shard} rooted at $v$, if it is the subgraph of $(V, \cE)$ spanned by $V_v$. We denote by $(V_v, \cE_v, \psi_{\restriction V_v})$ and $(V_v, \cE_v, \vec\psi_{\restriction V_v})$ its embedding in space-time and space respectively (see Fig.~\ref{fig:shards}).
\end{defi}

Thus, a shard is a set of $\s$ paths with distinct charges starting at a source. Definition~\ref{def:toomcontour} implies that the embedding of any path from a source other than the root to a sink is a nonempty sequence of diagonal segments followed by a possibly empty sequence of vertical segments. The same holds for the root, except that the sequence of diagonal edges might be empty, if the contour has only one sink. Therefore, it is easy to see that any Toom contour $(V, \cE, v_\circ, \psi)$ present in $X_0$ is uniquely determined by $v_\circ$ and the set of its embedded shards $\{(V_v, \cE_v, \psi_{\restriction V_v})\colon v\in V_\circ\}$.

We refer to vertices $w$ in a shard $(V_v, \cE_v)$ with $|\vec E_{\text{out}}(w)|=0$ as its \textit{endpoints}. We say that two embedded shards are \textit{connected}, if they have endpoints with identical space coordinates in their embedding. Note that the set of embedded shards of a Toom contour is connected.

We say that two embedded shards $(V,\cE,\psi)$ and $(V',\cE',\psi')$ are \emph{equivalent}, if there exists a bijection $\p:V\to V'$ such that it is an isomorphism between $(V,\vec E_s)$ and $(V',\vec E'_s)$ for all $s\in\S$ and $\vec\psi=\vec \psi'\circ\p$. That is, the two embedded shards are the same up to relabeling and time shift. We then say that two Toom contours are \emph{equivalent} if there is a bijection between their respective embedded shards such that each shard and its image are equivalent and the first contour's embedded root shard maps to the second one's. We will call the equivalence classes defined by this relation \textit{shattered contours}. See Fig.~\ref{fig:shards} for an example of the embedding of the shards of two equivalent Toom contours. We say that a shattered contour is \emph{rooted} at $o$, if the embedding of its root $v_\circ$ satisfies $\vec\psi(v_\circ)=o$.
\begin{defi}[Presence of a shattered contour]
	\label{def:shattered:present}
	A shattered contour is \emph{present} in $X_0$, if at least one Toom contour in the equivalence class is present in $X_0$.
\end{defi}

Putting our observations together, we obtain the following corollary of Theorem~\ref{T:contour}.
\begin{cor}[Presence of a shattered contour]
	\label{cor:shattered}
	If $o\in [X]$, a shattered contour rooted at $o$ is present in $X_0$.
\end{cor}

\begin{figure}[htb!]
	\begin{center}
		\includegraphics[width=10cm]{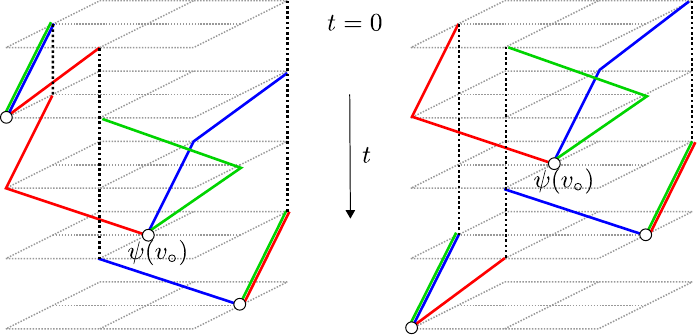}
		\caption{Space-time embedding of the set of shards of two Toom contours with three charges belonging to the same shattered contour as the one in Fig.~\ref{fig:toomcontour}. Empty dots denote the roots.}
		\label{fig:shards}
	\end{center}
\end{figure}

Note that by the definition of the equivalence relation and by conditions \ref{i}-\ref{xii} each shattered contour rooted at $o$ that is present in $X_0$ identifies with the space embedding of a connected set of shards, one of which is rooted at $o$, such that
\begin{enumerate}[label=(\roman*)',ref=(\roman*)']
	\item\label{i'} exactly one charge of each type arrives at the endpoints with identical $\vec\psi$ image,
	\item\label{ii'} $\vec\psi$ does not map together endpoints with other points,
	\item\label{iii'} each charge $s$ edge starting at an internal vertex or the root is a type $s$ diagonal segment,
	\item\label{iv'} every other edge is a diagonal segment with arbitrary type,
	\item\label{vi'} all non-root shards' sources are forks.
\end{enumerate}

\section{Peierls bounds}
We are now ready to apply a Peierls argument as in \cite{Swart22}, taking into account Section~\ref{subsec:shattered}. Let $\cT_{n,m}$ with $m, n\geq 0$ denote the set of shattered contours rooted at $o$ with $m+1$ shards and $n$ directed edges in their shards that start at an internal vertex or the root. Apart from these $n$ edges, there is exactly one edge of each charge starting at each of the $m$ sources other than the root, hence there is a total of $n+\s m$ edges in the shards of the shattered contours of $\cT_{n,m}$. Corollary~\ref{cor:shattered} provides the following starting point:
\be\label{Pei}
\P_p(o\in[X])
\leq \sum_{n=0}^\infty\sum_{m=0}^\infty\sum_{T\in\cT_{n,m}}\bbP_p\left(T\text{ is present}\right).
\ee
The following result is \cite[Lemma 12]{Swart22} that was first stated in \cite[Lemma 1]{Toom80}; we include the proof for completeness.

\bl[Zero sum property]
\label{L:zerosum}Recall the functions $L_s$ from Eq.~\eqref{eq:def:L}. If $(V,\Ei,v_\circ,\psi)$ is a Toom contour, then
\be\label{zerosum}
\sum_{s\in\S}\sum_{(v,w)\in\vec E_s}\left(L_s\left(\vec\psi(w)\right)-L_s\left(\vec\psi(v)\right)\right)=0.
\ee
\el
\bpro
We can rewrite the l.h.s.\ of Eq.~\eqref{zerosum} as
\be
\sum_{v\in V}\left\{\sum_{s\in\S}\sum_{(u,v)\in\vec E_{s,{\rm in}}(v)}L_s\left(\vec\psi(v)\right)
-\sum_{s\in\S}\sum_{(v,w)\in\vec E_{s,{\rm out}}(v)}L_s\left(\vec\psi(v)\right)\right\}.
\ee
At internal vertices, the term inside the brackets is zero because the number of incoming edges of each charge equals the number of outgoing edges of that charge. At the sources and sinks, the term inside the brackets is zero by Eq.~\eqref{eq:polar}, since there is precisely one outgoing (resp.\ incoming) edge of each charge.
\epro

The zero sum property is used in \cite[Lemma 13]{Swart22} and \cite[Lemma 1]{Toom80} to bound the total number of edges in a Toom contour. Since in our setting vertical segments give no contribution to Eq.~\eqref{zerosum}, we can instead bound the number of diagonal segments in terms of the number of sinks, or, equivalently, the number of edges in the shards of the corresponding shattered contour in terms of the number of shards.

\bl[Bound on the number of edges]
Let\label{L:edgebnd} $\eps$ and $R$ be as in~Eq.~\eqref{eq:def:R:eps}. Then each $T\in\cT_{n,m}$ satisfies $n\le Rm/\eps$.
\el

\begin{proof}
	By the linearity we have $L_s\left(\vec\psi(w)\right)-L_s\left(\vec\psi(v)\right)=L_s\left(\vec\psi(w)-\vec\psi(v)\right)$. Lemma~\ref{L:zerosum} and Eq.~\eqref{eq:BBPS:DTBP} and conditions~\ref{x} and \ref{xi} imply that
	\begin{align*}
	0 ={}&\sum_{s\in\S} \sum_{(v,w)\in \vec E_s}
	\left(L_s\left(\vec\psi(w)\right)-L_s\left(\vec\psi(v)\right)\right)\\
	={}&\sum_{s\in\S}\sum_{(v,w)\in  \vec E_s\setminus \vec E^\circ}
	\left(L_s\left(\vec\psi(w)\right)-L_s\left(\vec\psi(v)\right)\right)+\sum_{s\in\S}\sum_{(v,w)\in\vec E^\circ_s}
	\left(L_s\left(\vec\psi(w)\right)-L_s\left(\vec\psi(v)\right)\right)\\
	\geq{}&\eps n-Rm.\qedhere\end{align*}
\end{proof}

By condition \ref{i'} $\vec\psi$ maps the endpoints of the shards of any present $T\in\cT_{n,m}$ to $m+1$ disjoint sites in $X_0$. By Lemma~\ref{L:edgebnd}, we can then bound the sum in the r.h.s.\ of Eq.~\eqref{Pei} from above by
\begin{equation}
\label{Peierls}
\sum_{m=0}^\infty\sum_{n=0}^{Rm/\e}\sum_{T\in\cT_{n,m}}\P_p(T\text{ is present})
\le\sum_{m=0}^\infty\sum_{n=0}^{Rm/\e}\sum_{T\in\cT_{n,m}}p^{m+1}\le\sum_{m=0}^\infty p^m\sum_{n=0}^{Rm/\e} |\cT_{n,m}|.
\end{equation}
It then remains to bound the number of contours. As we have a more precise notion of contour, we count them differently from \cite[Lemma 13]{Swart22} or \cite[Lemma 3]{Toom80}, which allows for better bounds.

\bl[Exponential bound]
\label{lem:counting}
Recall $A=\bigcup_{s\in\S}A_s$. As $m\to\infty$
\begin{equation}
\label{eq:counting}
|\cT_{n,m}|\le \left(\max_{s\in\S}|A_s|^{R/\e }(2^{\s-1}-1)|A|(|A|-1)\frac{(R/\e+\s)^{R/\e+\s}}{\s^\s(R/\e)^{R/\e}}\right)^{m+o(m)}.
\end{equation}
\el
Before proving Lemma~\ref{lem:counting}, let us conclude the proof of our main results Theorems~\ref{T:bootstrap}, \ref{th:DTBP} and \ref{th:Toom}. 
\begin{proof}[Proof of Theorem~\ref{T:bootstrap}]
	By Lemma~\ref{lem:counting} the r.h.s.\ of Eq.~\eqref{Peierls} is finite for \begin{equation}
	\label{eq:pc:bound}
	p<\left(\max_{s\in\S}|A_s|^{R/\e }(2^{\s-1}-1)|A|(|A|-1)\frac{(R/\e+\s)^{R/\e+\s}}{\s^\s(R/\e)^{R/\e}}\right)^{-1}.
	\end{equation}
	By the Borel--Cantelli lemma, a.s.\ finitely many such shattered contours are present. Therefore, for $M$ large enough there is a positive probability that only contours with $m<M$ are present. But then, this event still occurs even if we remove from $X_0$ all sites at sufficiently large distance from the origin. Since this can decrease the probability that a shattered contour is present by at most some finite factor, we recover $\bbP_p(o\not\in[X_0])>0$. Hence, $\pc(\cU)$ is at least the r.h.s.\ of Eq.~\eqref{eq:pc:bound}, which is strictly positive.
\end{proof}

\begin{proof}[Proof of Theorem~\ref{th:DTBP}]
	Recall from~Eqs.~\eqref{eq:DTBPconstants} and \eqref{eq:DTBPsets} that for DTBP we have $\s=3$, $|A|=3$, $\e=1$, $R=6$ and $|A_s|=2$ for all $s\in\S$. Thus, the bound from Eq.~\eqref{eq:pc:bound} becomes 
	\[\pc(\cU^{\mathrm{DTBP}})\ge \left(2^{6}(2^{2}-1)3\cdot(3-1)\frac{(6+3)^{6+3}}{3^3\cdot 6^{6}}\right)^{-1}=\frac{1}{2\cdot3^{11}}> 2.8\cdot 10^{-6}.\qedhere\]
\end{proof}
\begin{proof}[Proof of Theorem~\ref{th:Toom}]
	By \cite[Proposition 3.1]{Hartarsky22sharpness} $\pc\left(\phi^{\mathrm{NEC}}\right)=\pc(\cU)$ for
	\[\cU:=\left\{\left\{(0,0,-1),(1,0,-1)\right\},\left\{(0,1,-1),(0,0,-1)\right\},\left\{(1,0,-1),(0,1,-1)\right\}\right\}.\]
	Upon applying an injective linear endomorphism of $\bbZ^3$, this is the same as
	\begin{equation}
	\label{eq:def:U'}\cU':=\left\{U\times\{-1\}\colon U\in\cU^{\mathrm{DTBP}}\right\}.\end{equation}
	Thus, we obtain the lower bound of Theorem~\ref{th:Toom} like Theorem~\ref{th:DTBP}, appending $-1$ in Eq.~\eqref{eq:DTBPsets} to all sites in $A_1,A_2,A_3$, appending $0$ to $u_1,u_2,u_3$ and changing nothing else.
\end{proof}
\begin{rem}
	At the price of degrading Theorems~\ref{th:DTBP} and \ref{th:Toom} to about $10^{-7}$, we could have used the simpler bound $|\cT_{n,m}|\le (2|A|)^{n+\s (m+1)}$, whose proof is left to the reader, instead of Lemma~\ref{lem:counting}. Inversely, examining \cite{Swart22} carefully, we may further improve the notion of fork to obtain $4.2\cdot 10^{-6}$, but this is hardly worth the effort. It is likely that one can make other minor improvements, but reaching, say, $10^{-3}$ with the present method seems hard.
\end{rem}

\begin{proof}[Proof of Lemma~\ref{lem:counting}]
	Recall that counting $\cT_{n,m}$ is equivalent to counting the space embeddings of $m+1$ connected shards with $n+\s m$ edges, one of which is rooted at $o$, and such that they satisfy conditions \ref{i'}-\ref{vi'}. Therefore, we may encode a shattered contour $T\in\cT_{n,m}$ in the following way. First, we supply a sequence of $m$ entries on the alphabet of all possible forks up to translation to specify the direction of the $\s$ edges from the sources (other than the root) subject to conditions \ref{iv'} and \ref{vi'}. Then we give a sequence of $n$ entries on an alphabet of $\max_{s\in\S} |A_s|$ elements called \emph{increments}, which specifies the direction of the segments corresponding to the $\s$ edges of the root and the edges starting at internal vertices, which by condition \ref{iii'} are elements of $A_1,\dots,A_\s$. Finally, we need $\s (m+1)-1$ separators to be inserted in the increment sequence.
	
	Given $T$ rooted at $o$, we determine this encoding as follows. We will process shards one by one, starting from the root one. In the case of the root shard, we explore the path of charge 1 from $v_\circ$ in the shard and register the increments $\vec\psi(w)-\vec\psi(v)\in A_1$ for edges $(v,w)$ in this path. To this purpose we have fixed an injective mapping from $A_1$ to the increment alphabet. Once we reach the endpoint of the path, we place a separator and repeat the same with the other $\s-1$ paths until the shard is exhausted. Up to this point we have registered $\s$ separators.
	
	During the entire process we keep track of a list of couples composed of the space coordinate $i$ of an endpoint and a charge $s\in\S$ in the following way. By condition \ref{i'}, the set of space coordinates of the endpoints contains $m+1$ distinct sites. As soon as we place a separator, we have either just discovered a new site in this set or we have rediscovered one. In the first case, we add to our list $\s-1$ couples corresponding to the space coordinate we discovered and the remaining charges (other than the one we used when discovering it). In the second case, we find the entry corresponding to the space coordinate and charge we used when rediscovering it and delete it from the list.
	
	In order to choose the second shard (and all the remaining ones), when the previous one is completely encoded, we read the first couple $(i, s)$ from the list. The next shard to encode is the one whose $s$-charge path ends at $i$. Once we know this, we register the source type of this shard, which is a fork by condition \ref{vi'}. We then explore its $\s$ paths exactly like we did for the root shard. When we reach an endpoint, we place a separator and either add $\s-1$ couples to our list or remove one as before.
	
	As $T$ consists of a connected sets of shards, this procedure ends when we have indeed encoded the entire shattered contour. It is clear from the construction that, given the encoding, we can reconstruct the space embedding of the shards, and thus the shattered contour. Indeed, we have ensured that we always now which charge of which shard we are reading, so that we can read off the corresponding increment from the encoding. Moreover, when we discover a new shard, we always know to which already discovered endpoint it should be connected in the space embedding and by which charge.
	
	It then remains to bound the number of possible encodings. By Lemma~\ref{L:edgebnd}, there are $\max_{s\in\S}|A_s|^n\leq \max_{s\in\S}|A_s|^{Rm/\e}$ choices for the increment sequence. By \ref{iv'} the size of the alphabet for forks is given by $(2^\s-2)|A|(|A|-1)/2$. Finally, the number of different ways in which we can insert $\s(m+1)-1$ separators into $n$ increments is
	\[\binom{n+\s(m+1)-1}{\s(m+1)-1}\le\binom{m(R/\e+\s)+\s-1}{m\s +\s-1}=\left(\frac{(R/\e+\s)^{R/\e+\s}}{\s^\s(R/\e)^{R/\e}}\right)^{m+o(m)}\] by Lemma~\ref{L:edgebnd}, as $m\to\infty$. Putting these together, we obtain Eq.~\eqref{eq:counting} as desired.
\end{proof}

\appendix
\section{Upper bound for the Toom rule}
In this appendix we sketch the proof of the upper bound in Remark~\ref{rem:Toom}, using \cite{Hartarsky21,Hartarsky22sharpness}. Recall $\mathcal U'$ from \eqref{eq:def:U'}, as well as the fact that $\pc\left(\phi^{\mathrm{NEC}}\right)=\pc(\cU')$ from the proof of Theorem \ref{th:Toom}. Therefore, it suffices to prove that $\pc(\mathcal U')<0.2452$, for which we closely follow \cite{Hartarsky21}. 

Fix $U'\in\mathcal U'$. By the correspondence of \cite[Proposition 3.1]{Hartarsky22sharpness} the bootstrap percolation model with update family $\{U'\}$ is equivalent to a standard two-dimensional oriented site percolation (for sites in state 1), but embedded in the plane $\langle U'\rangle$ generated by $U'$ in three-dimensional space. 

Now also fix a direction $u'\in S^2$ and consider the same percolation model, but restricted to $\bbZ^3\setminus\mathbb H_{u'}$. Observe that $\langle U'\rangle\setminus \mathbb H_{u'}$ is either $\langle U'\rangle$ or a half-plane thereof, depending on whether $u'\perp U'$. Denote by $o\to[A]\infty$ the event that in the model restricted to some $A\subset\bbZ^3$ the oriented percolation cluster of the origin (of sites in state 1) is infinite and define the critical parameter of this percolation
\[d_{u'}(\{ U'\})=\sup\left\{p\in[0,1]:\bbP_p\left(o\to[\langle U'\rangle\setminus \bbH_{u'}]\infty\right)>0\right\},\]
which we call the \emph{critical density}. Note that the critical density only depends on $u'$ via $\langle U'\rangle\setminus\mathbb H_{u'}$, which is constant along all open semicircles of $S^2$ with endpoints $U'^\perp$ (the meridians, if $U'^\perp\cap S^2$ are the poles). Therefore, it will suffice to consider $u'\in\bbR^2\times \{0\}$. But for such $u'$ we clearly have $d_{u'}(\{U'\})=d_{u'}(\{U\times\{0\}\})$, where $U$ is the update rule of DTBP such that $U'=U\times\{-1\}$.

It was proved in \cite[Sections 5, 6.1]{Hartarsky21} that 
\[\max_{u\in S^1} \min_{U\in\mathcal U^{\mathrm{DTBP}}}d_u(\{U\})<0.2452,\]
so that
\begin{equation}
\label{eq:app:lower}
\sup_{u'\in S^1\times\{0\}}\min_{U' \in\mathcal U'}d_{u'}(\{U'\})<0.2452.\end{equation}
Furthermore, for $u\in S^1$, $U\in\cU^{\mathrm{DTBP}}$ and $p>d_u(\{U\})$ more is known (see \cite[Section 5]{Hartarsky21}) about the oriented percolation cluster (of sites in state 1) starting at $o$. Namely, if we consider a wedge-shaped region $W=\bbH_{u+\varepsilon}\cap\bbH_{u-\varepsilon}$ with $\varepsilon=\varepsilon(p)$ small enough, then the radius of $\{x\in\bbZ^2\setminus W:o\to[\bbZ^2\setminus W]x\}$ has an exponentially decaying tail. This translates to an analogous result for $u'=u\times\{0\}$, $U'=U\times\{-1\}$ and the set $\{x\in \langle U'\rangle\setminus C_{u'},o\to[\langle U'\rangle\setminus C_{u'}]x\}$, where $C_{u'}=\bigcap_{v\in S^2:\|v-u'\|\le \varepsilon} \bbH_{v}$ is a cone-shaped region approximating $\bbH_{u'}$. Since $U'$ is fully contained in the lower half-space each oriented step we can take from the origin points towards the lower half-space, thus we may further truncate $C_{u'}$ to $C'_{u'}=C_{u'}\cap\mathbb H_{(0,0,1)}$.

Translating this into bootstrap percolation language, for $p> d_u(\{U\})=d_{u'}(\{U'\})$ the time at which the state of $o$ becomes 0 has an exponential tail for $\{U'\}$-bootstrap percolation with 0 boundary condition in $C'_{u'}$. Thus, the same holds for $\cU'$-bootstrap percolation. Taking advantage of  \eqref{eq:app:lower} for each $u'\in S^1\times\{0\}$ we can choose a set $U'\in\cU'$ such that $d_{u'}(\{U'\})<0.2452$. Hence, exponential decay holds for all $u'\in S^1\times \{0\}$ and $p\ge 0.2452$.

We next consider an elongated finite right circular cone with its vertex pointing in direction $(0,0,-1)$:
\[C=\left\{(x,y,z)\in\bbZ^3:K\sqrt{x^2+y^2}\le z+K^2,z\in[-K^2,0)\right\},\]
where $K>0$ is large enough. If $C$ is entirely in state state 0, it has probability exponentially close to 1 to extend further upwards. Namely, we claim that the sites $(x,y,0)\in\bbZ^2$ such that $x^2+y^2\le K^2$ are likely to turn to state 0 in $\cU'$-bootstrap percolation. Indeed, far from the boundary of this circle this is automatic, since the direction $(0,0,1)$ is not stable (as defined in Section~\ref{S:bootstrap}). Moreover, even at the site $Ku'$ at the boundary, up to distance of order $\varepsilon K$ from it, we see more sites in state 0 than those in $K u'+C'_{u'}$. Therefore, we do a union bound with the exponential decay estimates, in order to obtain that with probability exponentially close to $1$ in $\varepsilon K$ the entire circle becomes state 0. Repeating the same procedure for the slightly taller cone we just obtained and noting that the exponential bounds are summable, we get that the cone has positive probability to grow indefinitely. By standard renormalisation arguments this yields that $\pc(\mathcal U')<0.2452$, as desired. We refer the reader to \cite[Section 4]{Hartarsky21} for more details, but let us remark that the existence of a $K$ such that $C$ contains the relevant parts of the truncated cones uniformly in the choice of $u$ follows from the continuity of critical densities of oriented percolation models (see \cite[Remark 4.4 and Lemma 5.1]{Hartarsky21}). 

\section*{Acknowledgments}
We thank Cristina Toninelli for helpful and stimulating discussions. We thank Rob Morris for information regarding \cite{Balister22}.

\let\d\oldd
\let\k\oldk
\let\l\oldl
\let\L\oldL
\let\o\oldo
\let\O\oldO
\let\r\oldr
\let\S\oldS
\let\t\oldt
\let\u\oldu

\bibliographystyle{plain}
\bibliography{Bib}
\end{document}